\documentclass[11pt]{amsart} 
\usepackage[margin=1.5in]{geometry}

\usepackage{amssymb,upgreek}
\usepackage{amsmath}
\usepackage{url}
\usepackage{bm}
\usepackage[pagebackref]{hyperref}

\usepackage{enumerate}
\usepackage{xcolor}

\usepackage{tikz}

\usepackage{color}
\usepackage{mathtools} 
\usepackage{caption}
\usepackage{subcaption}
\usepackage{fancyhdr}
\usepackage{mathrsfs}
\newtheorem{theorem}{Theorem}[section] 
\newtheorem{lemma}[theorem]{Lemma}

\newtheorem{corollary}[theorem]{Corollary}
\newtheorem*{theorem*}{Theorem}
\newtheorem*{lemma*}{Lemma}
\newtheorem*{corollary*}{Corollary}

\theoremstyle{definition}
\newtheorem{definition}[theorem]{Definition}
\newtheorem{remark}[theorem]{Remark}

\newtheorem*{remark*}{Remark}
\newtheorem*{definition*}{Definition}
\newtheorem*{remarks*}{Remarks}
\newtheorem*{addenda*}{Addenda}

\DeclareMathOperator{\fix}{Fix}

\newcommand{\cpone}{\mathbb{CP}^1}

\newcommand{\cptwobar}{\smash{\overline{\mathbb{CP}}{}^2}}

\title[Exotic definite four-manifolds]
{Exotic definite four-manifolds with\\ non-cyclic fundamental group}

\author[R. Harris]{Robert Harris}
\address{Department of Pure Mathematics, 
University of Waterloo, 
Waterloo, ON, N2L 3G1, Canada} 
\email{robert.harris@uwaterloo.ca}

\author[P. Naylor]{Patrick Naylor}
\address{Department of Mathematics and Statistics,
McMaster University, 
Hamilton, ON, L8S 4K1, Canada}
\email{patrick.naylor@mcmaster.ca}

\author[B. D. Park]{B. Doug Park}
\address{Department of Pure Mathematics, 
University of Waterloo, 
Waterloo, ON, N2L 3G1, Canada} 
\email{bdpark@uwaterloo.ca}

\date{May 15, 2024}



\begin{document}

\begin{abstract}
We construct infinitely many pairwise non-diffeomorphic smooth structures on a definite $4$-manifold with non-cyclic fundamental group $\mathbb{Z}/2\times \mathbb{Z}/2$.
\end{abstract}

\maketitle

\section{Introduction}
Throughout this paper, a $4$-manifold will mean a closed connected oriented smooth $4$-dimensional manifold.  
We say that a $4$-manifold $X$\/ has an \emph{exotic}\/ smooth structure if $X$\/ possesses more than one smooth structure, i.e., there exists a $4$-manifold $X'$\/ that is homeomorphic but not diffeomorphic to $X$. In this paper, $G$\/ will always denote the product group 
$\mathbb{Z}/2\times \mathbb{Z}/2$, the non-cyclic group of order four. 

The first example of an exotic smooth structure on a $4$-manifold with a definite intersection form was given by Levine, Lidman and Piccirillo in \cite{levlidpic23}.  
Their construction yielded 4-manifolds with fundamental group $\mathbb{Z}/2$. 
To distinguish the smooth structures, they used explicit handle structures to compute the Ozsv\'ath–Szab\'o closed 4-manifold invariant. 
Later, Stipsicz and Szab\'o constructed more definite examples with $\mathbb{Z}/2$ fundamental group in \cite{stisza23} and \cite{stisza23-2} as the quotient spaces of free $\mathbb{Z}/2$ actions on certain simply connected $4$-manifolds with exotic smooth structures.  

Indefinite examples with odd $b_2^+$ and various finite cyclic fundamental groups were constructed by Torres in \cite{tor11} and \cite{tor14}. More recently, examples with $\mathbb{Z}/2$ fundamental group and even $b_2^+$ were constructed by Beke, Koltai and Zampa in \cite{bekkolzam23}.  
In this paper, we will construct infinitely many exotic smooth structures on a definite $4$-manifold with fundamental group isomorphic to  $G=\mathbb{Z}/2\times\mathbb{Z}/2$.  
More precisely we will prove:  

\begin{theorem}\label{theorem: main}
There exists a\/ $4$-manifold $Q$\/ with the following properties.
\begin{itemize}
    \item[(i)] $\pi_1(Q)\cong \mathbb{Z}/2\times\mathbb{Z}/2$;
    \item[(ii)]  $b_2(Q)=4$ and the intersection form on $Q$\/ is negative definite;  
    \item[(iii)]  $Q$\/ possesses infinitely many pairwise non-diffeomorphic smooth structures.  
\end{itemize}
\end{theorem}

Our exotic smooth structures will be obtained as free quotients of 
homotopy K3 surfaces, i.e., $4$-manifolds that are homeomorphic but non-diffeomorphic to the complex K3 surface. 
A 4-manifold that is obtained as a free $\mathbb{Z}/2\times \mathbb{Z}/2$ quotient of the K3 surface is called an \emph{Enriques-Einstein-Hitchin}\/ manifold or an \emph{EEH}\/ manifold for short. 
In this paper, we will work with two alternate descriptions of an EEH manifold: as the quotient of a complete intersection in $\mathbb{CP}^5$, and as a quotient of a double branched cover of $\cpone\times \cpone$. 
We will use these descriptions to construct \emph{generalized}\/ 
Enriques-Einstein-Hitchin manifolds, which are free $\mathbb{Z}/2\times \mathbb{Z}/2$ quotients of homotopy K3 surfaces.  

EEH manifolds have been studied by both algebraic and differential geometers for some time, and a comprehensive reference book on the subject is \cite{degitekha00}.  
By a theorem of Hitchin in \cite{hit74}, if $X$\/ is an Einstein $4$-manifold, then its signature $\sigma(X)$ and its Euler characteristic $\chi(X)$ satisfy $|\sigma(X)|\leq \frac{2}{3}\chi(X)$. 
Moreover, the equality occurs exactly for $4$-manifolds which are either flat or one of three (trivial, $\mathbb{Z}/2$, or $\mathbb{Z}/2\times \mathbb{Z}/2$) quotients of the K3 surface. 

Our basic strategy for constructing exotica is as follows.  
Let $X$\/ be a complex K3 surface with a holomorphic elliptic fibration $f:X\to \mathbb{CP}^1$. 
Suppose that there is a free smooth action of $G$\/ on $X$\/ that preserves the elliptic fibration (i.e., each element of $G$\/ maps a fiber of $f$\/ to a fiber of $f$). 
Starting with a smooth torus fiber $T$\/ whose orbit under the $G$\/ action consists of four disjoint torus fibers, we perform the same Fintushel-Stern knot surgery (cf.~\cite{finste98}) four times along each of these four fibers using the same knot in a certain family of knots $\{K_m \mid m\in \mathbb{Z}_+ \}$ that have distinct Alexander polynomials.  
If $X_{K_m}$ denotes the resulting $4$-manifold, then $X_{K_{m_1}}$ and $X_{K_{m_2}}$ will be pairwise homeomorphic but pairwise non-diffeomorphic when $m_1\neq m_2$.  
By the work of Hambleton and Kreck in \cite{hamkre88}, the corresponding quotient spaces $\{X_{K_m}/G \mid m\in\mathbb{Z}_+ \}$\/ will then necessarily contain an infinite family of $4$-manifolds that are homeomorphic but pairwise non-diffeomorphic.

\subsection*{Organization}

In \S\ref{sec: smooth structures}, we review the basic material from \cite{finste98} and \cite{finste09} on the Fintushel-Stern knot surgeries and their Seiberg-Witten invariants, and show how to smoothly distinguish the generalized Enriques-Einstein-Hitchin manifolds.  
In \S\ref{sec: hitchin example}, we review an algebro-geometric description of a particular EEH $4$-manifold, and in \S\ref{sec: branched cover}, we give another description using a standard construction of the K3 surface as a double branched cover.

\subsection*{Acknowledgements} 
The second and third authors were supported by NSERC Discovery Grants. The second author was also supported by a grant from McMaster University.  We thank Ian Hambleton for many helpful discussions, and in particular, telling us about Hitchin's construction in \cite{hit74}. We would also like to thank Tyrone Ghaswala for enlightening discussions about branched covers.

\section{Distinguishing smooth structures}\label{sec: smooth structures}

In this section, we briefly review the knot surgery operation due to Fintushel and Stern, and how it changes the Seiberg-Witten invariants of a 4-manifold.  
Supposing that we can find a free $\mathbb{Z}/2\times\mathbb{Z}/2$ action on the K3 surface that preserves the elliptic fibration, we will show that this leads to infinitely many exotic smooth structures on generalized Einstein-Enriques-Hitchin $4$-manifolds. 

\begin{definition}
    Let $X$ be a $4$-manifold containing an embedded torus $T$\/ with trivial normal bundle, and suppose that $K$\/ is a knot in $S^3$. 
    The result of \emph{knot surgery}\/ along $T$\/ is a $4$-manifold of the form 
    \[X_K := (X- \nu(T)) \cup_{\varphi} (S^1\times (S^3- \nu(K)), \]
    where $\nu(T)\cong T\times D^2$ is a tubular neighborhood of $T$, 
    and the gluing map $\varphi$\/ between the boundary $3$-tori sends the homology class of the meridian $\mu(T) = \{\textnormal{point}\}\times \partial D^2\subset \partial(\nu(T))$ of $T$\/ to that of the longitude of $K$. 
\end{definition}

Note that the above definition does not completely determine the isotopy class of the gluing map $\varphi$, but this is not always necessary. 
If $X$ is simply connected and $\pi_1(X\setminus T)=1$, then $X_K$ is also simply connected and has the same intersection form as $X$.  
By Freedman's Theorem in \cite{fre82}, $X$ and $X_K$ are homeomorphic. For further details regarding this construction, the reader is referred to \cite{finste98}.

Now suppose that $X$\/ is a $4$-manifold with $b_2^+(X)>1$.  
Recall that the Seiberg-Witten invariant of $X$\/ can be expressed as an integer-valued function
\[
SW_X: H^2(X;\mathbb{Z})\longrightarrow \mathbb{Z}
\]
that has finite support.  
If $h:X_1\to X_2$ is a diffeomorphism, then 
\[
SW_{X_1}(h^{\ast}(L))=\pm SW_{X_2}(L)
\]
for all $L\in H^2(X_2;\mathbb{Z})$.  
If we write $H=H^2(X;\mathbb{Z})$, then 
the Seiberg-Witten invariant of $X$\/ can also be expressed as an element 
\[
\overline{SW}_X=\sum_{L\in H} SW_X(L)\, L 
\in \mathbb{Z}[H], 
\]
in the integer group ring of $H$.  
For each positive integer $m>0$, we let $K_m$ be a knot with symmetrized Alexander polynomial equal to
\[
\Delta_{K_m}(t)=mt-(2m-1)+mt^{-1}.
\]
For example, we could take $K_m$ to be the twist knot with $2m+1$ half twists. The following lemma can be easily derived from the works of Fintushel and Stern in \cite{finste98} and \cite{finste09}.

\begin{lemma}\label{lemma: SW-invariant}
{\rm (\S1 of Lecture~3 in \cite{finste09})}
Let $X$ be a\/ $4$-manifold with $b_2^+(X)>1$ and a nontrivial Seiberg-Witten invariant $SW_X\not\equiv 0$.  
Let $T_i$ $(i=1,\dots,r)$ be disjoint smoothly embedded tori in $X$\/ that lie in the same non-torsion homology class $[T_i]=[T]\in H_2(X;\mathbb{Z})$ with square $[T]^2=0$.  
Also assume that\/ $\pi_1(X)=1$ and $\pi_1(X-T_i)=1$ for all\/ $i=1,\dots,r$.  
Let $X_m$\/ denote the result of performing a knot surgery along each $T_i$ all using the same knot $K_m$.  
Then we have 
\[
\overline{SW}_{X_m} = \overline{SW}_X \cdot (\Delta_{K_m}(PD(2[T])))^r,
\]
where $PD:H_2(X;\mathbb{Z})\to H^2(X;\mathbb{Z})$ denotes the Poincar\'{e} duality homomorphism and the product on the right-hand side is the product inside the group ring $\mathbb{Z}(H^2(X_m;\mathbb{Z}))$.  
\end{lemma}

In particular, let $X$\/ be an elliptic K3 surface.  
Note that ${SW}_X(0)=1$ and $SW_X(L)=0$ for all $L\neq 0\in H^2(X;\mathbb{Z})$.  
If $T$\/ denotes a smooth torus fiber of an elliptic fibration  $f:X\to\mathbb{CP}^1$, then 
$\pi_1(X-T)=1$ since $\pi_1(X)=1$ and there is a sphere section of $f$\/ that will bound any meridian circle of $T$.   
Lemma~\ref{lemma: SW-invariant} then implies that 
\[
\overline{SW}_{X_m} = \big(\Delta_{K_m}(PD(2[T]))\big)^r
= \big(mPD(2[T])-(2m-1)[0]+m(-PD(2[T]))\big)^r,  
\]
where $[0]\in H^2(X;\mathbb{Z})$ denotes the trivial class
and the exponent means that we take the $r$-fold product in the group ring.  
By comparing the coefficients of $\overline{SW}_{X_m}$, we immediately see that $X_m$'s consist of pairwise non-diffeomorphic $4$-manifolds.  

Now assume that there exists a free orientation-preserving action of $$G=\mathbb{Z}/2\times\mathbb{Z}/2=\langle\sigma,\tau\mid \sigma^2=\tau^2=1,\;\sigma\tau=\tau\sigma\rangle$$ on our elliptic K3 surface $X$.  
Assume furthermore that our $G$\/ action preserves the elliptic fibration, i.e., each fiber is mapped into a fiber.  
Let $T_1$\/ be a generic torus fiber of $X$\/ such that its $G$\/ orbit 
consists of four disjoint smooth tori:  
$T_1$, $T_2=\sigma(T_1)$, $T_3=\tau(T_1)$ and $T_4=(\tau\circ\sigma)(T_1)$.  
Choose tubular neighborhoods $\nu(T_i)$ of $T_i$ in $X$\/ ($i=1,\ldots,4$) that are also disjoint. 
We need to perform a Fintushel-Stern knot surgery on each of $T_1,\ldots, T_4$ using the same knot $K_m$ equivariantly so that our free $G$\/ action on $X-\sqcup_{i=1}^4\nu(T_i)$ extends to the resulting homotopy K3 surface $X_m$.  

Fix the positive integer $m$.  
Let $E=S^1 \times (S^3-\nu(K_m))$, where $S^3-\nu(K_m)$ denotes the complement of the tubular neighborhood $\nu(K_m)$ of the knot $K_m$ in $S^3$.  
We take four copies of $E$, denoted by $E_1,\ldots, E_4$.  
To obtain $X_m$, we glue $E_1,\ldots, E_4$ to the complement $X-\sqcup_{i=1}^4\nu(T_i)$: 
\[
X_m = \left(X-\sqcup_{i=1}^4 \nu(T_i)\right) \cup_{\varphi_i} \left(\sqcup_{i=1}^4 E_i\right), 
\]
where $\varphi_i: \partial(E_i) \to \partial(\nu(T_i))$ are the gluing diffeomorphisms on the boundary $3$-tori.  

To specify $\varphi_i$, we first choose a framing of the boundary component $\partial(\nu(T_1))$, i.e., a diffeomorphism 
$$j_1: T^3=S^1_\alpha\times S^1_\beta \times S^1_\mu \longrightarrow \partial(\nu(T_1)),$$ 
such that each torus $(S^1_\alpha\times S^1_\beta)\times \ast\,$ is mapped to a fiber $T_1^{\parallel}$ that is parallel to $T_1$, and 
each third-factor circle $(\ast\times\ast)\times S^1_\mu$ is mapped to a meridian of $T_1$.  
Since $\sigma$, $\tau$ and $\tau\circ\sigma$\/ all preserve the torus fibers, the compositions $\sigma\circ j_1$, $\tau\circ j_1$ and $(\tau\circ\sigma)\circ j_1$\/ are framings of $\partial(\nu(T_2))$, 
$\partial(\nu(T_3))$ and $\partial(\nu(T_4))$, respectively.  

Next we choose the framing on the boundary $\partial E$, i.e., a diffeomorphism  
$$j:T^3=S^1_\alpha\times S^1_\beta \times S^1_\mu \longrightarrow \partial E$$ 
such that each first-factor circle $S^1_\alpha \times(\ast\times\ast)$ is mapped to a circle $S^1\times \{\mathrm{point}\}$, where the point lies on the boundary $\partial (S^3-\nu(K_m))$, 
each second-factor circle $\ast \times S^1_\beta\times\ast$ is mapped to a meridian $\mu(K)$, and 
each third-factor circle $(\ast\times\ast) \times S^1_\mu$ is mapped to a longitudinal knot $\lambda(K)$.  
We use the same framing $j$\/ for each of the four boundary components $\partial E_1,\ldots, \partial E_4$.  

Now we are ready to specify the gluing diffeomorphisms $\varphi_i: \partial(E_i) \to \partial(\nu(T_i))$.  
We will choose $\varphi_1=j_1\circ j^{-1}$, $\varphi_2=(\sigma\circ j_1)\circ j^{-1}$, $\varphi_3=(\tau\circ j_1)\circ j^{-1}$, and $\varphi_4=(\tau\circ\sigma\circ j_1)\circ j^{-1}$.
We can check immediately that every $\varphi_i$ maps each torus $S^1\times\mu(K)$ to a parallel fiber $T_i^{\parallel}$ and maps each longitudinal knot $\lambda(K)$ to a meridian of $T_i$.  
Thus every $\varphi_i$ defines a Fintushel-Stern knot surgery.  
We also immediately see that the free $G$\/ action on the complement $X-\sqcup_{i=1}^4 \nu(T_i)$ extends to a free $G$\/ action on $X_m$ by extending by appropriate identity maps.  
The action of $\sigma$\/ can be extended by the identity maps $E_1\to E_2$, $E_2\to E_1$, $E_3\to E_4$, and  $E_4\to E_3$.   
The action of $\tau$\/ can be extended by the identity maps $E_1\to E_3$, $E_3\to E_1$, $E_2\to E_4$, and  $E_4\to E_2$.   
The action of $\tau\circ\sigma$\/ can be extended by the identity maps $E_1\to E_4$, $E_4\to E_1$, $E_2\to E_3$, and  $E_3\to E_2$.   

It follows that $G$\/ acts freely on each homotopy K3 surface $X_m$.  
Since $X_m$ is simply connected, the corresponding quotient space $Q_m=X_m/G$\/ has fundamental group that is isomorphic to $G=\mathbb{Z}/2\times\mathbb{Z}/2$.
Recall that the Euler characteristic and the signature of $X$ (and hence those of $X_m$) are $24$ and $-16$, respectively.  
The Euler characteristic of $Q_m$ is then equal to the quotient $24/4=6$.  
Since $b_1(Q_m)=0$, we must have $b_2(Q_m)=4$.  
By Hirzebruch's signature theorem (see Theorem 8.2.2 on p.~86 of \cite{hir66}), the signature is multiplicative over unbranched covers.  
Thus the signature of the quotient space $Q_m$ is equal to $-16/4=-4=-b_2(Q_m)$.  
It follows that $Q_m$ has a negative definite intersection form.
By a generalization of Donaldson's diagonalization theorem in \cite{don83} to a non-simply connected setting (e.g. Theorem 2.4.18 in \cite{nic00}), the intersection form of $Q_m$ is given by $\oplus^{4}\langle -1 \rangle$.

To show that there are exotic smooth structures, we need to recall the following theorem due to Hambleton and Kreck.  

\begin{theorem}\label{theorem: hambleton-kreck}
{\rm (Corollary to (1.1) on p.~87 of \cite{hamkre88})} 
Let $n\in\mathbb{Z}$ and let $G$\/ be a finite group.
Then there are only finitely many homeomorphism types among all $4$-manifolds whose Euler characteristic is equal to $n$ and whose fundamental group is isomorphic to $G$.  
\end{theorem}

\begin{corollary}
The collection $\{Q_m\mid m\in\mathbb{Z}_+\}$ contains infinitely many homeomorphic $4$-manifolds that are pairwise non-diffeomorphic.  
\end{corollary}

\begin{proof}
If there was a diffeomorphism $h: Q_{m_1}\to Q_{m_2}$, then we could lift $h$\/ to a diffeomorphism $\widetilde{h}: X_{m_1}\to X_{m_2}$ between the universal covers, which is a contradiction.  
Hence $\{Q_m\mid m\in\mathbb{Z}_+\}$ consists of pairwise non-diffeomorphic $4$-manifolds.  
But by Theorem~\ref{theorem: hambleton-kreck} and the pigeonhole principle, infinitely many of the $Q_m$'s must be homeomorphic. 
\end{proof}

\begin{remark}
As far as the authors know, there is no homeomorphism classification of the $4$-manifolds whose fundamental group is $\mathbb{Z}/2\times\mathbb{Z}/2$ at this time. 
However, by work of Kasprowski, Powell and Ruppik in \cite{kaspowrup20}, the homotopy type (but not the homeomorphism type) of such a 4-manifold is determined by its quadratic 2-type.
Consequently, we cannot conclude that each $Q_m$ is homeomorphic to some Enriques-Einstein-Hitchin manifold. In a sequel paper, we hope to determine the homomorphism type of $Q_m$. 

\end{remark}

\section{First example}\label{sec: hitchin example}

In this section, we will present a concrete example of a K3 surface with a free $\mathbb{Z}/2\times \mathbb{Z}/2$ action.  
Our example was first studied in detail by Hitchin in \cite{hit74} (p.~440).  
Let $A=[A_{ij}]$ and $B=[B_{ij}]$ be real $3\times 3$ matrices, and 
let $x=(x_1,x_2,x_3)$, $y=(y_1,y_2,y_3)\in\mathbb{C}^3$.  
If $A$\/ and $B$\/ are invertible, then the three homogeneous quadratic equations
\begin{align}
A_{11}x_1^2+A_{12}x_2^2+A_{13}x_3^2
+B_{11}y_1^2+B_{12}y_2^2+B_{13}y_3^2&=0, \nonumber \\
A_{21}x_1^2+A_{22}x_2^2+A_{23}x_3^2
+B_{21}y_1^2+B_{22}y_2^2+B_{23}y_3^2&=0, 
\label{eq: complete intersection}\\
A_{31}x_1^2+A_{32}x_2^2+A_{33}x_3^2
+B_{31}y_1^2+B_{32}y_2^2+B_{33}y_3^2&=0 \nonumber
\end{align}
define a complete intersection variety $X$\/ in $\mathbb{CP}^5$.  

It can be shown that $X$\/ is a K3 surface (cf.~Exercise~1.3.13(e) on p.~24 of \cite{gomsti99}).  
It was also observed in \cite{hit74} that if 
\begin{equation}\label{eq: conditions for free action}
A_{1,j}>0, \ B_{1,j}>0, \ A_{2,j}>0, \textrm{\ and\ } -B_{2,j}>0
\end{equation}
for all $j=1,2,3$, then 
\begin{equation*}
\sigma(x,y)=(\overline{x},\overline{y}), \quad
\tau(x,y)=(x, -y) 
\end{equation*}
define commuting involutions that generate a free $\mathbb{Z}/2\times \mathbb{Z}/2$ action on $X$. 
On the other hand, an elliptic fibration structure on $X$\/ was not described in \cite{hit74}.  
We will now define an elliptic fibration on $X$\/ that is preserved by this $\mathbb{Z}/2\times \mathbb{Z}/2$ action.  

We start by observing that the sign conditions in (\ref{eq: conditions for free action}) do not involve the third rows of $A$\/ and $B$.  
Let us choose 
\begin{equation}\label{eq: choice of third row}
A_{31}=B_{32}=1, \ A_{32}=B_{31}=-1, \textrm{\ and\ } A_{33}=B_{33}=0. 
\end{equation} 
There are many such matrices $A$\/ and $B$\/ that satisfy both (\ref{eq: conditions for free action}) and (\ref{eq: choice of third row}).  
For example, we could choose 
\begin{equation*}
A= \left[\begin{array}{rrr}
1&2&1\\
1&1&1\\
1&-1&\hspace{5pt} 0
\end{array}\right], \ 
B= \left[\begin{array}{rrr}
1&2&1\\
-1&-1&-1\\
-1&1&0
\end{array}\right],
\end{equation*}
which satisfy $\det(A)=\det(B)=1$.  
Note that the third equation in (\ref{eq: complete intersection}) now becomes 
\begin{equation}\label{eq: choice of third equation}
x_1^2-x_2^2-y_1^2+y_2^2=(x_1+y_1)(x_1-y_1) - (x_2+y_2)(x_2-y_2) =0.   
\end{equation}

Next we consider the complete intersection variety $E_{(\lambda:\mu)}$\/ in $\mathbb{CP}^5$ given by the four equations:  
\begin{equation}\label{eq: torus equations}
\begin{array}{rl}
A_{11}x_1^2+A_{12}x_2^2+A_{13}x_3^2
+B_{11}y_1^2+B_{12}y_2^2+B_{13}y_3^2&=0, \\[3pt]
A_{21}x_1^2+A_{22}x_2^2+A_{23}x_3^2
+B_{21}y_1^2+B_{22}y_2^2+B_{23}y_3^2&=0, \\[3pt]
\lambda(x_1+y_1)-\mu(x_2-y_2)&=0, \\[3pt]
\mu(x_1-y_1)-\lambda(x_2+y_2)&=0,
\end{array}
\end{equation}
where the first two equations in (\ref{eq: torus equations}) are exactly the same as in (\ref{eq: complete intersection}) and $(\lambda:\mu)\in\mathbb{CP}^1$. 
The last two linear equations in (\ref{eq: torus equations}) together imply that equation (\ref{eq: choice of third equation}) holds for all points in $E_{(\lambda:\mu)}$.  
It follows that $E_{(\lambda:\mu)}$\/ is a subset of our K3 surface $X$.  

Note that any single point $(x,y)$ in $E_{(\lambda:\mu)}$ can be used to determine the ratio $\lambda/\mu$. 
Hence every point $(x,y)$ of $X$\/ lies in $E_{(\lambda:\mu)}$ for a unique $(\lambda:\mu)\in \mathbb{CP}^1$.  
Let $f:X\to \mathbb{CP}^1$ be defined by $f(x,y)=(\lambda:\mu)$, where $(x,y)\in E_{(\lambda:\mu)}$.  
The $2\times 6$ coefficient matrix of the last two linear equations in (\ref{eq: torus equations}) is 
\begin{equation*}
\left[\begin{array}{rrrrrr}
\lambda & -\mu & \ 0 & \lambda & \mu &  \ 0 \\
\mu & -\lambda & \ 0 & -\mu & -\lambda & 0
\end{array}
\right],
\end{equation*}
which has rank $2$ for all $(\lambda:\mu)\in \mathbb{CP}^1$.  
Thus each fiber $f^{-1}(\lambda:\mu)=E_{(\lambda:\mu)}$ is a complex curve.  
It is known (cf.~Exercise~3.7(iii) on p.~152 of \cite{dim92}) that the genus of a generic complete intersection curve of multi-degree $(d_1,\ldots,d_{n-1})$ in $\mathbb{CP}^n$ is
\[
g=1-\frac{1}{2}d_1\cdots d_{n-1}\left(n+1-\sum_{i=1}^{n-1}d_i\right).
\]
Since $E_{(\lambda:\mu)}$ has multi-degree $(2,2,1,1)$ in $\mathbb{CP}^5$, we conclude that $E_{(\lambda:\mu)}$ has genus $g=1$ for generic $(\lambda:\mu)$.  
Hence we have shown that $f$\/ is an elliptic fibration.  

We now verify that the $\mathbb{Z}/2\times \mathbb{Z}/2$ action on $X$\/ maps fibers of $f$\/ to fibers of $f$.  
Suppose that $(x,y)$ lies in $E_{(\lambda:\mu)}=f^{-1}(\lambda:\mu)$, i.e., $(x,y)$ satisfies the four equations in (\ref{eq: torus equations}).  
By taking the complex conjugates of all terms in (\ref{eq: torus equations}), we can see that $\sigma(x,y)=(\overline{x},\overline{y})$ lies in $E_{(\overline{\lambda}:\overline{\mu})}=f^{-1}(\overline{\lambda}:\overline{\mu})$.  
By changing $y$\/ to $-y$, we can see that the coefficients $\lambda$\/ and $\mu$ switch their roles, and hence $\tau(x,y)=(x, -y)$ lies in $E_{(\mu:\lambda)}=f^{-1}(\mu:\lambda)$.  
Similarly, $(\tau\circ\sigma)(x,y)=(\overline{x},-\overline{y})$ lies in $E_{(\overline{\mu}:\overline{\lambda})}=f^{-1}(\overline{\mu}:\overline{\lambda})$.  

Moreover, $\sigma$, $\tau$ and $\tau\circ \sigma$ all map a generic fiber of $f$\/ to another fiber of $f$.  
The only exceptions are as follows:
\begin{itemize}
    \item[(i)] $\sigma$\/ maps $f^{-1}(\lambda:\mu)$ to itself when $\lambda\overline{\mu}\in\mathbb{R}$. 
    
    \item[(ii)] $\tau$\/ maps $f^{-1}(\lambda:\mu)$ to itself when
    $(\lambda:\mu)=(1:1)$ or $(\lambda:\mu)=(1:-1)$.
    
    \item[(iii)] $\tau\circ\sigma$\/ maps $f^{-1}(\lambda:\mu)$ to itself when $|\lambda|=|\mu|$.  
    
\end{itemize}
Hence, using \S\ref{sec: smooth structures} we may perform the same knot surgeries along four distinct torus fibers in a generic orbit, e.g., along 
$f^{-1}(1:1+i)$, $f^{-1}(1:1-i)$, $f^{-1}(1+i:1)$, and $f^{-1}(1-i:1)$.

\section{Second example}\label{sec: branched cover}

In this section, we present another construction of the Enriques-Einstein-Hitchin manifold $E$\/ which is more reminiscent of \cite{stisza23} and would be equally convenient for our purposes. Let $Y=\mathbb{CP}^1\times\mathbb{CP}^1$, and consider the maps $s,c:\cpone\to \cpone$ given by
\[s:(u:v)\mapsto (-u:v) \quad\quad\quad\quad c:(u:v)\mapsto (\bar{v}:\bar{u}).\]

Note that $s$ corresponds to a rotation of the 2-sphere with two fixed points, $c$ is an involution with a fixed circle, and $s\circ c$ has no fixed points. Define automorphisms of $Y$ by
\[r = s\times s \quad\quad\quad\quad j = c\times (s\circ c)\]
and observe that 
\begin{itemize}
    \item[(i)] $r$, $j$, and $r\circ j$ are commuting automorphisms of order two;
    \item[(ii)] $r$ is holomorphic and has exactly four fixed points;
    \item[(iii)] $j$\/ and $r\circ j$\/ are antiholomorphic and fixed point free.
\end{itemize}

Now consider the standard construction (see e.g.\ \S7.3 of \cite{gomsti99}) of the K3 surface as the desingularization of the double branched cover of $Y$\/ over a reducible curve of bidegree $(4,4)$. 
To build this concretely, we start with the subset
\begin{equation}\label{eq: curve C}
    C=\left(\,\cup_{i=1}^4 (\mathbb{CP}^1\times\{p_i\})\right)\cup\left(\, \cup_{j=1}^4(\{q_j\}\times\mathbb{CP}^1)\right),
\end{equation}
and choose the points $\{p_i\}$ and $\{q_j\}$ so that $C$\/ is preserved by all three of the maps $r$, $j$, and $r\circ j$, and also so that $C\cap \fix(r)=\emptyset$. 
Explicitly, for each collection, one may take the points 
$$\{(1:1+i),\, (-1:1+i),\, (1-i:1),\, (-1+i:1) \}.$$ 

The subset $C$\/ is not an embedded curve, but we may blow up the 16 transverse intersection points to obtain a smooth double branched cover $X\to Y\#16\cptwobar$ that is branched over the proper transform $\widetilde{C}$ of $C$\/ under the blow-ups.  
The resulting covering space $X$\/ is a K3 surface. 

To obtain $E$, we will take the quotient of $X$\/ by a free $\mathbb{Z}/2\times \mathbb{Z}/2$ action obtained by lifting the maps $r$ and $j$. First, we observe that both $r$ and $j$ extend to maps on the blow-up.

\begin{lemma}\label{lemma: blow up extension}
    The map $r$ extends to a holomorphic involution of $Y\#16\cptwobar$ with four fixed points. The maps $j$ and $r\circ j$ extend to antiholomorphic fixed point free involutions of $Y\#16\cptwobar$.    
\end{lemma}

\begin{proof}
    For any self-intersection point $a=(q_j,p_i)$ of $C$, let $\pi:Y\#2\cptwobar\rightarrow Y$\/ be the blow-down map corresponding to the two blow-ups at $a$ and $r(a)$. 
    Since the restriction $Y\#2\cptwobar-\pi^{-1}(\{a,r(a)\})\rightarrow Y-\{a,r(a)\}$ is biholomorphic, it follows that an extension of $r$ to $Y\#2\cptwobar$ is determined by an involution on $\pi^{-1}(\{a,r(a)\})=\pi^{-1}(a)\cup \pi^{-1}(r(a))\cong \cpone\sqcup \cpone$. 
    
    Consider a ball $D\subset Y$\/ containing $a$\/ that is small enough such that $D$ and $r(D)$ are disjoint, do not contain any other intersection point of $C$\/ aside from $a$ and $r(a)$ respectively, and for which $D \cap \fix(r)=r(D)\cap \fix(r)=\emptyset$. 
    Now, for a point $h\in \pi^{-1}(a)$, let $H\subset D$\/ be any {smooth curve} passing through $a$\/ such that its proper transform of $\widetilde{H}$ intersects $\pi^{-1}(a)$ at $h$. 
    Since $r$ is a rotation (and hence holomorphic) it follows that $r(H)$ is a smooth curve in $r(D)$ and so the set $\widetilde{r(H)}\cap \pi^{-1}(r(a))$ contains exactly one point, which we call $h'$. 
    One can also check that starting with $h'$ and applying the same process yields $h$ as the corresponding element in $\widetilde{H}\cap \pi^{-1}(a)$.
    Therefore, the map that interchanges $h$\/ with $h'$ defines an involution $r_1$ on $\pi^{-1}(\{a,r(a)\})$. 
    It follows that extending $r$\/ by $r_1$, we obtain a holomorphic involution $r_1'$ on $Y\#2\cptwobar$. 
    
    Furthermore, since $r_1$ is fixed point free, the fixed points of $r_1'$ are in one-to-one correspondence with the fixed points of $r$, and so $r_1'$ has exactly four fixed points, which are the images of the four fixed points of $r$ in the blow-up $Y\#2\cptwobar$. 
    Proceeding inductively we obtain an extension of $r$ to $Y\#16\cptwobar$ (still denoted $r'$) with the desired properties. 
    One defines the extensions $j'$ and $(r \circ j)'$ in a similar manner.
\end{proof}

Next we will lift these maps to smooth involutions defined on $X$. 
The following lemma is straightforward to prove, but we include a proof for the convenience of the reader.

\begin{lemma}\label{lemma: lifting lemma}{\rm (cf. \S3.1 of \cite{degkha97} or Lemma 2.1 of \cite{stisza23})} 
Suppose that $X$ and $Y$ are\/ $4$-manifolds and that\/ $b:X\to Y$ is a\/ $2$-fold branched covering map with the branch locus $C\subset Y$. Suppose that $f:Y\to Y$ is a smooth involution and that:
\begin{itemize}
    \item[(i)] $f$ preserves $C$ set-wise;
    \item[(ii)] $\fix(f)\cap C=\emptyset$;
    \item[(iii)] $f_*$ commutes with the representation of the branched covering map \\
    $\phi:H_1(Y-C;\mathbb{Z})\to \mathbb{Z}/2$.
\end{itemize}
Then there is a lift $\tilde{f}:X\to X$ of $f$ which is fixed point free.
\end{lemma}

\begin{proof}
    For a point $x_0\in X$ with $b(x_0)\notin C$, define $\tilde{f}(x_0)$ by choosing one of the lifts of $f(b(x_0))$. 
    Once this choice is made, we can define the rest of the lift as follows: for any other point $x\in X$ with $b(x)\notin C$, choose a path $\gamma:x_0\to x$. 
    Then $f(b(\gamma))$ is a path from $f(b(x_0))$ to $f(b(x))$. 
    We define $\tilde{f}(x)$ to be the endpoint of a lift of $f(b(\gamma))$ starting at $\tilde{f}(x_0)$. 
    Note that this endpoint does not depend on our choice of $\gamma$. 
    Indeed, if $\eta$ is any other such path, the loop $b(\gamma*\eta^{-1})$ lifts to a loop if and only if $f(b(\gamma*\eta^{-1}))$ does, since $f_*$ commutes with $\phi$. 
    If $b(x)\in C$, then we can unambiguously define $\tilde{f}(x) = b^{-1}(f(b(x)))$. 
    Since $b\circ \tilde{f}^2 = f^2\circ b=b$, it follows that $\tilde{f}^2$ preserves the fibers of $b$, and so $\tilde{f}$ has order either two or four.  
    
    Now, if $z\in X$ and $b(z)$ is a fixed point of $f$ (note $b(z)\notin C$ by assumption), the lift $\tilde{f}$ either preserves or exchanges the two lifts of $b(z)$. We claim that if $\tilde{f}(z) = z$, then $\tilde{f}(z')=z'$ for all other points $z'$ with $b(z')$ a fixed point of $f$. Indeed, choose a path $\gamma$ from $z$ to $z'$; then $f(z')$ is the endpoint of a lift of $b(\gamma)$ starting from $z$. 
    However, $b(\gamma)$ is still a path from $b(z)$ to $b(z')$. 
    Similar to the above argument, the (well-defined) lift of $b(\gamma)$ starting from $z$\/ is exactly $\gamma$. 
    Thus, by composing with the deck transformation if necessary, we conclude that $f$\/ has exactly one fixed point free lift. 
\end{proof}

\begin{remark}\label{remark: order is 2 or 4}
    In the case that the lift of $f$ has order $2$, one obtains a fixed point free involution on $X$. In general, the lift may have order $2$ or $4$ (see \S 2.3 of \cite{stisza23}). 
\end{remark}

\begin{lemma}
    The maps $r$, $j$, and  $(r\circ j)$ lift to a free\/  $\mathbb{Z}/2\times \mathbb{Z}/2$ action on X. 
\end{lemma}

\begin{proof}
    We will verify the hypotheses of Lemma~\ref{lemma: lifting lemma} for the branch curve $\widetilde{C}$, the proper transform of (\ref{eq: curve C}). First, let $\pi:Y\#16\cptwobar\rightarrow Y$\/ be the blow-down map, and 
    let $E=\bigcup_{i,j}\pi^{-1}(q_j,p_i)$ be the union of the exceptional $2$-spheres in $Y\#16\cptwobar$. 
    Let $r':Y\#16\cptwobar\to Y\#16\cptwobar$ be as in the proof of Lemma~\ref{lemma: blow up extension}.  
    Then for a point $x\in \pi^{-1}(C)-E$, we have by construction that 
    $$r'(x)=\left(\pi^{-1}|_{(Y\#16\cptwobar)-E}\circ r\circ \pi\right)(x).$$ 
    Furthermore, as $\pi(x)\in C$\/ and $C$\/ is preserved by $r$, it follows that 
    $$r'(x)\in \pi^{-1}|_{(Y\#16\cptwobar)-E}(C)\subset \widetilde{C}.$$ 
    On the other hand, if $x\in \widetilde C\cap E$, then $x\in \widetilde C\cap \pi^{-1}(q_j,p_i)$ for some $j,i$. 
    Moreover, since the only components of $C$\/ that contain $(q_j,p_i)$ are $\cpone\times \{p_i\}$ and $\{q_j\}\times \cpone$, it follows that $x$ is either an element of (the single point in) $\widetilde{\cpone\times \{p_i\}}\cap \pi^{-1}(q_j,p_i)$ or $\widetilde{\{q_j\}\times\cpone}\cap \pi^{-1}(q_j,p_i)$. It particular, we have that one of $r'(x)\in \widetilde{r(\cpone\times \{p_i\})}\cap \pi^{-1}r(q_j,p_i)$ or $r'(x)\in \widetilde{r(\{q_j\}\times\cpone)}\cap \pi^{-1}r(q_j,p_i)$ is true. 
    Regardless, since both $r(\{q_j\}\times\cpone)$ and $r(\cpone\times \{p_i\})$ are contained in $C$ it must be the case that $r'(x)$ is contained in $\widetilde C$. 
    Therefore $\widetilde C$ is preserved by $r'$. 
    An analogous argument shows that $j$ (and $r\circ j$) also preserves $\widetilde C$ as a set.

    By Lemma \ref{lemma: blow up extension}, the set $\fix(r')$ is exactly $\pi^{-1}|_{(Y\#16\cptwobar)-E}(\fix(r))$. So, we have
    \begin{align*}
        \fix(r')\cap \widetilde C &= \pi^{-1}|_{(Y\#16\cptwobar)-E}(\fix(r))\cap\widetilde C\\
        &\subseteq \pi^{-1}(\fix(r))\cap\pi^{-1}(C)\\
        &=\pi^{-1}(\fix(r)\cap C) =\emptyset.
    \end{align*}
    Now, $H_1(Y\#16\cptwobar-\widetilde C;\mathbb Z)\cong \mathbb Z$\/ 
    since each exceptional $2$-sphere $\pi^{-1}(q_j,p_i)$ gives rise to a cylinder that connects the meridians of $\widetilde{\cpone\times \{p_i\}}$ and $\widetilde{\{q_j\}\times\cpone}$.  
    As $r'$, $j'$ and $(r\circ j)'$ are automorphisms, $r_\ast'$, $j_\ast'$ and $(r\circ j)_\ast'$ are plus or minus the identity map on the group $\mathbb{Z}$, and hence commute with $\phi:H_1(Y\#16\cptwobar-\widetilde C;\mathbb Z)\rightarrow\mathbb Z/2$. 
    Therefore, by Lemma~\ref{lemma: lifting lemma}, $r'$, $j'$ and $(r\circ j)'$ all lift to free actions on the K3 surface $X$. 
    We will denote these lifts by $\tilde r$, $\tilde j$ and $\widetilde {r\circ j}$,  respectively. 

    By Remark~\ref{remark: order is 2 or 4}, each of these three lifts has order $2$ or $4$.  
    By an index-theoretic argument, Hitchin showed (see the last three paragraphs of \S3 in \cite{hit74}) that a K3 surface cannot support a free $\mathbb{Z}/4$ action.  
    Thus all of our lifts have order two.  
    The fact that $\tilde{r}$ has order $2$ was also observed in the first paragraph of \S3.1 in \cite{degkha97}. 
    Since $\tilde r\circ \tilde j = \widetilde{r\circ j}$ and $\tilde j\circ \tilde r = \widetilde{j\circ r}$, we see that $\tilde r$ and $\tilde j$ commute because $r$\/ and $j$ commute.   
    It follows that the subgroup 
    $\langle \tilde r,\tilde j\rangle\subset \textrm{Aut} (X)$ is isomorphic to $\mathbb Z/2\times \mathbb Z/2$. 
\end{proof}

A particular {Einstein-Enriques-Hitchin} manifold is now given by the quotient space $E:=X/\langle \tilde r,\tilde j\rangle$. Since the maps $r$\/ and $j$\/ preserve the rulings of $Y$, their lifts will preserve the elliptic fibration (that is induced by the projection map onto the second factor, $\mathrm{pr}_2:Y=\cpone\times\cpone\to\cpone$).  
As in \S \ref{sec: smooth structures}, we can now perform  Fintushel-Stern knot surgeries along four disjoint torus fibers related by this action.

\bibliographystyle{abbrv}
\bibliography{references}

\begin{thebibliography}{10}

\bibitem{bekkolzam23}
M.~Beke, L.~Koltai, and S.~Zampa.
\newblock New exotic four-manifolds with $\mathbb{Z}/2\mathbb{Z}$ fundamental group.
\newblock {\em {\rm arXiv:2312.08452}}, 2023.

\bibitem{degitekha00}
A.~Degtyarev, I.~Itenberg, and V.~Kharlamov.
\newblock {\em Real {E}nriques Surfaces}, volume 1746 of {\em Lecture Notes in Math.}
\newblock Springer-Verlag, Berlin, 2000.

\bibitem{degkha97}
A.~Degtyarev and V.~Kharlamov.
\newblock Real {E}nriques surfaces without real points and {E}nriques-{E}instein-{H}itchin 4-manifolds.
\newblock In {\em The {A}rnoldfest $(${T}oronto, {ON}, $1997)$}, volume~24 of {\em Fields Inst. Commun.}, pages 131--140. Amer. Math. Soc., Providence, 1999.

\bibitem{dim92}
A.~Dimca.
\newblock {\em Singularities and Topology of Hypersurfaces}.
\newblock Universitext. Springer-Verlag, New York, 1992.

\bibitem{don83}
S.~K. Donaldson.
\newblock An application of gauge theory to four-dimensional topology.
\newblock {\em J. Differential Geom.}, 18(2):279--315, 1983.

\bibitem{finste98}
R.~Fintushel and R.~J. Stern.
\newblock Knots, links, and {$4$}-manifolds.
\newblock {\em Invent. Math.}, 134(2):363--400, 1998.

\bibitem{finste09}
R.~Fintushel and R.~J. Stern.
\newblock Six lectures on four 4-manifolds.
\newblock In {\em Low Dimensional Topology}, volume~15 of {\em IAS/Park City Math. Ser.}, pages 265--315. Amer. Math. Soc., Providence, 2009.

\bibitem{fre82}
M.~H. {Freedman}.
\newblock The topology of four-dimensional manifolds.
\newblock {\em J. Differential Geom.}, 17(3):357--453, 1982.

\bibitem{gomsti99}
R.~E. Gompf and A.~I. Stipsicz.
\newblock {\em {$4$}-Manifolds and Kirby Calculus}, volume~20 of {\em Grad. Stud. Math.}
\newblock Amer. Math. Soc., Providence, 1999.

\bibitem{hamkre88}
I.~Hambleton and M.~Kreck.
\newblock On the classification of topological 4-manifolds with finite fundamental group.
\newblock {\em Math. Ann.}, 280(1):85--104, 1988.

\bibitem{hir66}
F.~Hirzebruch.
\newblock {\em Topological Methods in Algebraic Geometry}, volume 131 of {\em Die Grundlehren der mathematischen Wissenschaften}.
\newblock Springer-Verlag, New York, 3rd edition, 1966.

\bibitem{hit74}
N.~Hitchin.
\newblock {Compact four-dimensional Einstein manifolds}.
\newblock {\em J. Differential Geom.}, 9(3):435--441, 1974.

\bibitem{kaspowrup20}
D.~{Kasprowski}, M.~{Powell}, and B.~{Ruppik}.
\newblock {Homotopy classification of 4-manifolds with finite abelian 2-generator fundamental groups}.
\newblock {\em {\rm arXiv:2005.00274}}, 2020.

\bibitem{levlidpic23}
A.~S. Levine, T.~Lidman, and L.~Piccirillo.
\newblock New constructions and invariants of closed exotic 4-manifolds.
\newblock {\em {\rm arXiv:2307.08130}}, 2023.

\bibitem{nic00}
L.~I. Nicolaescu.
\newblock {\em Notes on {S}eiberg-{W}itten Theory}, volume~28 of {\em Grad. Stud. Math.}
\newblock Amer. Math. Soc., Providence, 2000.

\bibitem{stisza23-2}
A.~I. Stipsicz and Z.~Szab\'o.
\newblock Definite four-manifolds with exotic smooth structures.
\newblock {\em {\rm arXiv:2310.16156}}, 2023.

\bibitem{stisza23}
A.~I. {Stipsicz} and Z.~{Szab\'o}.
\newblock {Exotic definite four-manifolds with non-trivial fundamental group}.
\newblock {\em {\rm arXiv:2308.08388}}, 2023.

\bibitem{tor11}
R.~Torres.
\newblock Geography of spin symplectic four-manifolds with abelian fundamental group.
\newblock {\em J. Aust. Math. Soc.}, 91(2):207--218, 2011.

\bibitem{tor14}
R.~Torres.
\newblock Geography and botany of irreducible non-spin symplectic 4-manifolds with abelian fundamental group.
\newblock {\em Glasg. Math. J.}, 56(2):261--281, 2014.

\end{thebibliography}

\end{document}